\documentclass[12pt]{amsart}
\usepackage{enumerate}
\usepackage{amsmath, amssymb, amsthm, amsfonts}
\usepackage{geometry}
\usepackage{bbm} %for the 1 symbol of indicator functions
\usepackage{url}

\newtheorem{introthm}{Theorem}
\newtheorem{theorem}{Theorem}[section]

\newtheorem{lemma}[theorem]{Lemma}

\newtheorem*{claim*}{Claim}

\theoremstyle{definition}
\newtheorem{remark}[theorem]{Remark}
\newtheorem*{remark*}{Remark}

\newtheorem{definition}[theorem]{Definition}

\newtheorem{example}[theorem]{Example}

\def\E{\mathbb{E}}
\def\P{\mathbb{P}}

\def\IR{\mathbb{R}}
\def\IZ{\mathbb{Z}}
\def\IN{\mathbb{N}}

\def\eps{\varepsilon}
\def\al{\alpha}

\def\ga{\gamma}
\def\Ga{\Gamma}

\DeclareMathOperator{\Aut}{Aut}

\DeclareMathOperator{\var}{var}
\DeclareMathOperator{\cov}{cov}
\DeclareMathOperator{\corr}{corr}

\DeclareMathOperator{\dist}{dist}
\DeclareMathOperator{\sign}{sign}

\newcommand{\defeq}{\mathrel{\vcenter{\baselineskip0.5ex \lineskiplimit0pt
                     \hbox{\scriptsize.}\hbox{\scriptsize.}}}%
                     =}

\def\ind{\mathbbm{1}} %the 1 symbol of indicator functions

\begin{document}

\title{Mutual information decay for factors of IID}

\author[Gerencs\'{e}r]{Bal\'{a}zs Gerencs\'{e}r}
\address{MTA Alfr\'ed R\'enyi Institute of Mathematics 
H-1053 Budapest, Re\'altanoda utca 13-15;
and ELTE E\"otv\"os Lor\'and University, Department of Probability and Statistics 
H-1117 Budapest, P\'azm\'any P\'eter s\'et\'any 1/c}
\email{gerencser.balazs@renyi.mta.hu}

\author[Harangi]{Viktor Harangi}
\address{MTA Alfr\'ed R\'enyi Institute of Mathematics 
H-1053 Budapest, Re\'altanoda utca 13-15}
\email{harangi@renyi.hu}

\thanks{The first author was supported by 
NKFIH (National Research, Development and Innovation Office) grant PD 121107. 
The second author was supported 
by Marie Sk{\l}odowska-Curie Individual Fellowship grant no.\ 661025 
and the MTA R\'enyi Institute ``Lend\"ulet'' Groups and Graphs Research Group.}

\keywords{factor of IID, factor of Bernoulli shift, mutual information, entropy inequality}

\subjclass[2010]{37A35, 60K35, 37A50}

%\date{}

%
\begin{abstract}
This paper is concerned with factor of i.i.d.\ processes on the $d$-regular tree for $d \geq 3$. 
We study the mutual information of the values on two given vertices. 
If the vertices are neighbors (i.e., their distance is $1$), then 
a known inequality between the entropy of a vertex and the entropy of an edge 
provides an upper bound for the (normalized) mutual information. 
In this paper we obtain upper bounds for vertices at an arbitrary distance $k$, of order $(d-1)^{-k/2}$. 
Although these bounds are sharp, we also show that an interesting phenomenon occurs here: 
for any fixed process the rate of decay of the mutual information 
is much faster, essentially of order $(d-1)^{-k}$.
\end{abstract}

\maketitle

%%%%%%%%%%%%%%%%%%%%%%%%%%%%%%%%%%%%%%%%%%%%%%%%%%%%%%%%%%%%%%%%%%%%%%%%%%%%%%%%%%%%%%%%%%%%%
%%                                                                                         %%
%%        INTRODUCTION                                                                     %%
%%                                                                                         %%
%%%%%%%%%%%%%%%%%%%%%%%%%%%%%%%%%%%%%%%%%%%%%%%%%%%%%%%%%%%%%%%%%%%%%%%%%%%%%%%%%%%%%%%%%%%%%

\section{Introduction}

For an integer $d \geq 3$ let $T_d$ denote the $d$-regular tree:
the (infinite) connected graph with no cycles 
and with each vertex having exactly $d$ neighbors. 

This paper deals with \emph{factor of i.i.d.\ processes} on $T_d$. 
First we give an informal definition: independent and identically distributed
(say $[0,1]$ uniform) random labels are assigned to the vertices of $T_d$,
then each vertex gets a new label that depends on
the labeled rooted graph as seen from that vertex,
all vertices ``using the same rule''. 

For a formal definition, let $V(T_d)$ denote the vertex set and
$\Aut(T_d)$ the automorphism group of $T_d$. 
Suppose that $M$ is a measurable space. 
(In most cases $M$ will be either a discrete set or $\IR$.) 
A measurable function $F \colon [0,1]^{V(T_d)} \to M^{V(T_d)}$ 
is said to be an \emph{$\Aut(T_d)$-factor} (or \emph{factor} in short)
if it is $\Aut(T_d)$-equivariant, that is,
it commutes with the natural $\Aut(T_d)$-actions. 
Given an i.i.d.\ process $Z = \left( Z_v \right)_{v \in V(T_d)}$
on $[0,1]^{V(T_d)}$, applying $F$ yields a factor of i.i.d.\ process $X = F(Z)$,
which can be viewed as a collection $X = \left( X_v \right)_{v \in V(T_d)}$ of
$M$-valued random variables. It follows immediately from the definition
that the distribution of $X$ is invariant under the action of $\Aut(T_d)$;
in particular, each $X_v$ has the same distribution.
Factors of i.i.d.\ are also studied by ergodic theory 
(under the name of \emph{factors of Bernoulli shifts}), 
see Section \ref{sec:2} for details. 

One of the reasons why factor of i.i.d.\ processes have attracted a growing attention
in recent years is that they give rise to certain randomized local algorithms. 
Suppose that we have a finite $d$-regular graph that locally looks like $T_d$, 
that is, around most vertices the neighborhoods are trees (up to some large radius). 
Then i.i.d.\ labels can be put on the vertices 
and a given factor mapping can be applied (approximately) at each vertex, 
yielding a randomized algorithm on the finite graph. 
The distribution of the random output of this algorithm 
is described locally by the factor of i.i.d.\ process. 
See \cite{endreuj, E, V, hoppen} for how such local algorithms 
can be used to obtain large independent sets. 
(Whether a graph is locally tree-like 
is related to the number of cycles. 
The \emph{girth} of a graph is the length of its shortest cycle. 
When we say that a finite graph has \emph{large essential girth}, 
we mean that the number of short cycles is small compared to the number of vertices. 
Around most vertices of such a graph the neighborhoods are trees up to a large radius. 
Note that \emph{random regular graphs} have large essential girth with high probability.) 

The starting point of our investigations is 
the following entropy inequality 
which holds for any factor of i.i.d.\ process $X$ with a finite state space $M$: 
\begin{equation} \label{eq:edge-vertex}
H( X_u, X_v ) \geq \frac{2(d-1)}{d} H( X_v ) \mbox{, where $uv$ is an edge.}
\end{equation}
Here $H( X_v )$ is the (Shannon) entropy of the discrete random variable $X_v$, 
and $H( X_u, X_v )$ stands for the joint entropy of $X_u$ and $X_v$, 
see Section \ref{sec:entropy} for the definitions. 
(Because of the $\Aut(T_d)$-invariance the distribution of $X_v$ is the same for each vertex $v$. 
Similarly, the joint distribution of $(X_u,X_v)$ is the same for any edge $uv$.) 
Rahman and Vir\'ag proved \eqref{eq:edge-vertex} in a special setting \cite{mustazeebalint}. 
A full and concise proof was given by Backhausz and Szegedy 
in \cite{invtree}; see also \cite{mustazee}. 
The counting argument behind this inequality 
actually goes back to a result of Bollob\'as 
on the independence ratio of random regular graphs \cite{bollobas}. 
As we will see in Section \ref{sec:freegroup-factors}, 
a more general version of \eqref{eq:edge-vertex} can also 
be found implicitly (for even $d$) in Lewis Bowen's work on free group actions \cite{bowen}. 

Entropy inequalities played a central role in a couple of remarkable results recently: 
the Rahman-Vir\'ag result \cite{mustazeebalint} 
about the maximal size of a factor of i.i.d.\ independent set on $T_d$
and the Backhausz-Szegedy result \cite{ev}
on eigenvectors of random regular graphs.

The inequality \eqref{eq:edge-vertex} can also be expressed 
as an upper bound for the mutual information 
of two neighboring vertices $u$ and $v$: 
\begin{equation} \label{eq:mut_info}
\frac{I(X_u;X_v)}{H(X_v)} \leq \frac{2}{d} .
\end{equation}
Recall that the mutual information $I(X_u;X_v)$ is defined as 
$H(X_u) + H(X_v) - H(X_u,X_v)$ and can be viewed as 
(the expected value of) the information 
gained about one of the random variables knowing the other one. 
In our case the random variables are identically distributed, 
therefore they have the same entropy $H(X_u) = H(X_v)$. 
Dividing the mutual information by this entropy 
results in a \emph{normalized mutual information} 
which measures the amount of shared information 
proportional to the total amount of information. 
This ratio is always between $0$ and $1$, 
and being close to $0$ intuitively means 
that the random variables are ``almost independent''. 
(It is reasonable to normalize the mutual information this way, see Example \ref{ex:tuple}.) 
 
A natural question arises: what can be said about the mutual information of 
two vertices $u$ and $v$ at distance $k$? One expects that the mutual information 
tends to $0$ as the distance grows. But what is the rate of decay? 
We get very different answers depending on how the question is posed exactly. 

First let us consider the problem for a fixed $k \geq 1$, 
that is, we look for a universal upper bound for 
the normalized mutual information $I(X_u; X_v) / H(X_v)$ 
that holds for any factor of i.i.d.\ process with a finite state space $M$. 
The following bounds are obtained.  
\begin{introthm} \label{thm:universal_bound}
Let $M$ be a finite state space and $d \geq 3$. 
For any $u,v \in V(T_d)$ at distance $k$ 
and for any factor of i.i.d.\ process $X$ on $M^{V(T_d)}$ we have 
\begin{equation} \label{eq:univ_bound}
\frac{I(X_u;X_v)}{H(X_v)} \leq \begin{cases}
\frac{2}{d(d-1)^l} & \mbox{ if $k=2l+1$ is odd,} \\
\frac{1}{(d-1)^l} & \mbox{ if $k=2l$ is even.}
\end{cases}
\end{equation}
These bounds are the best possible in the sense that for any fixed $k$ 
there exist factor of i.i.d.\ processes for which 
the normalized mutual information tends to the bound above. 
\end{introthm}
According to \eqref{eq:univ_bound}, the normalized mutual information 
for distance $k$ is (at most) of order $(1/\sqrt{d-1})^k$, and this is sharp. 
However, it turns out that there does not exist a single factor of i.i.d.\ process 
that would show the sharpness of the bound for all $k$ at once. 
In fact, for any fixed process the mutual information decays at a much faster rate, 
basically of order $1/(d-1)^k$. 
\begin{introthm} \label{thm:fixed_process}
Let $M$ be a finite state space and $d \geq 3$. 
If $X = \left( X_v \right)_{v \in V(T_d)}$ is a factor of i.i.d.\ process 
on $M^{V(T_d)}$, then 
\begin{equation}
I(X_u;X_v) \leq \frac{|M| (k+1)^2}{(d-1)^k} ,
\end{equation}
where $|M|$ denotes the cardinality of $M$ (number of states).  
\end{introthm}
This bound is essentially sharp, see Example \ref{ex:fixed_process}. 

\subsection*{Motivation}
Our motivation to study this problem is multi-fold. 
On the one hand, many aspects of independence in factors of i.i.d.\ have been studied earlier 
(e.g.\ correlation for real-valued processes or 
triviality of various tail $\sigma$-algebras). 
Our goal was to get a quantitative result about 
how much independence these processes exhibit when $M$ is finite. 
Mutual information has the advantage over correlation 
that the latter only detects linear dependence. 
On the other hand, we aimed to obtain new entropy inequalities. 
The edge-vertex inequality \eqref{eq:edge-vertex} and its \emph{blow-ups} 
(where both the vertex $v$ and the edge $uv$ are replaced 
with all the vertices in their respective $R$-radius neighborhoods) 
have a number of applications already. 
Theorem \ref{thm:universal_bound} is a generalization of \eqref{eq:edge-vertex}, 
and as such one expects it will provide further applications. 

\subsection*{Proof methods}
To prove Theorem \ref{thm:universal_bound} we will consider the $d$-regular tree $T_d$ 
as the Cayley graph of different groups $G$ depending on the parity of $d$ and $k$. 
When $k$ is even, we will use the free product $G = \IZ_2 \ast \cdots \ast \IZ_2$. 
When $k$ is odd, either $G = \IZ \ast \cdots \ast \IZ$ (for even $d$) or  
$G = \IZ \ast \cdots \ast \IZ \ast \IZ_2$ (for odd $d$) will be used. 
In each case we will try to find as many elements in $G$ as possible 
such that they freely generate a subgroup and 
each element has length $k$ (w.r.t.~the corresponding word metric in $G$). 
In other words, we will look for a maximum-rank free subgroup $H \leq G$ 
that has a generating set consisting of elements with length $k$. 
Once we have such a free subgroup $H$, Theorem \ref{thm:universal_bound} will 
follow from a more general version of the edge-vertex entropy inequality 
(Theorem \ref{thm:free_group_entropy_inequality}). 
This inequality is known from Lewis Bowen's work on free group actions, 
namely it is equivalent to the fact that the \emph{$f$-invariant} is non-negative 
for factors of the Bernoulli shift \cite{bowen}. 

Theorem \ref{thm:fixed_process} will be deduced from 
the correlation decay result of Backhausz, Szegedy, and Vir\'ag \cite{cordec}, 
which says that for a real-valued factor of i.i.d.\ process ($M=\IR$) 
the correlation of two vertices $u$ and $v$ at distance $k$ 
is (at most) of order $1 / (\sqrt{d-1})^k$. 
In the case of a finite state space $M$, 
by assigning a real number to each state we can replace 
our original process with a real-valued one. 
Consequently, for any assignment $M \to \IR$ the correlation bound 
tells us something about the joint distribution of $X_u$ and $X_v$ (for the original process). 
The idea is to try to find suitable %appropriate?
assignments that yield a good bound on the mutual information of $X_u$ and $X_v$. 

\subsection*{Outline of the paper}
The rest of the paper is structured as follows.
In Section \ref{sec:2} we go through basic definitions 
and explain the more general entropy inequality we will need 
to prove the universal bound. 
The proofs of Theorem \ref{thm:universal_bound} and \ref{thm:fixed_process} 
are given in Section \ref{sec:3} and \ref{sec:4}, respectively.  
Finally, in Section \ref{sec:5} we present examples showing 
that the above theorems are (essentially) sharp. 

\subsection*{Acknowledgments}
We are grateful to \'Agnes Backhausz, Bal\'azs Szegedy, B\'alint Vir\'ag, 
and M\'at\'e Vizer for fruitful discussions on the topic. 
We would also like to thank the anonymous referee 
for many valuable comments and suggestions. 

%%%%%%%%%%%%%%%%%%%%%%%%%%%%%%%%%%%%%%%%%%%%%%%%%%%%%%%%%%%%%%%%%%%%%%%%%
%%%%%%%%%%%%%%%%%%%%%%%%%%%%%%%%%%%%%%%%%%%%%%%%%%%%%%%%%%%%%%%%%%%%%%%%%
%%%%%%%%%%%%%%%%%%%%%%%%%%%%%%%%%%%%%%%%%%%%%%%%%%%%%%%%%%%%%%%%%%%%%%%%%

\section{Preliminaries} \label{sec:2}

\subsection{Entropy and mutual information} \label{sec:entropy}
Let $X$ be a discrete random variable 
taking $m$ distinct values with probabilities $p_1, \ldots, p_m$. 
Then the \emph{Shannon entropy} of $X$ is defined as 
$$ H(X) \defeq \sum_{i=1}^m -p_i \log( p_i ) .$$
Given two discrete random variables $X$ and $Y$, 
$(X,Y)$ can be considered as a discrete random variable itself, 
and its entropy is denoted by $H(X,Y)$. 
(This is often called the \emph{joint entropy} of $X$ and $Y$.) 
One can define the \emph{mutual information} of $X$ and $Y$ by 
$$ I( X;Y ) \defeq H(X) + H(Y) - H(X,Y) .$$ 
Another way to define mutual information is via \emph{conditional entropies}: 
$$ I( X;Y ) = H(X) - H(X | Y) = H(Y) - H(Y | X) ,$$
where the conditional entropy $H(X|Y) = H(X,Y) - H(Y)$ can be expressed 
as the expectation (in $Y$) of the entropy of 
the (conditional) distribution of $X$ conditioned on $Y$, that is, 
\begin{equation*} 
H(X | Y) = \sum_{j=1}^n \P(Y = y_j) 
\sum_{i=1}^m -\P(X=x_i \, | \, Y=y_j) \log \P(X=x_i \, | \, Y=y_j) ,
\end{equation*}
where $x_1, \ldots, x_m$ and $y_1, \ldots, y_n$ 
denote the values taken by $X$ and $Y$, respectively. 
In other words, if $f_i$ denotes the mapping $y \mapsto \P( X=x_i | Y=y)$, 
then 
\begin{equation} \label{eq:conditional_entropy}
H(X | Y) = \E \sum_{i=1}^m -f_i(Y) \log f_i(Y) .
\end{equation}

\subsection{Factors of i.i.d.}

Although the results of this paper concern $\Aut(T_d)$-factors, 
we will need to use the notion of factors in a more general setting. 
Suppose that a group $\Ga$ acts on a countable set $S$.
Then $\Ga$ also acts on the space $M^S$ for a set $M$:
for any function $f \colon S \to M$ and for any $\ga \in \Ga$ let
\begin{equation} \label{eq:action}
(\ga \cdot f)(s) \defeq f( \ga^{-1} \cdot s) \quad \forall s \in S .
\end{equation}
First we define the notion of factor maps.
\begin{definition}
Let $M_1,M_2$ be measurable spaces and
$S_1,S_2$ countable sets with a group $\Ga$ acting on both.
A measurable mapping $F \colon M_1^{S_1} \to M_2^{S_2}$ is
said to be a \emph{$\Ga$-factor} if it is $\Ga$-equivariant,
that is, it commutes with the $\Ga$-actions.
\end{definition}
By an \emph{invariant (random) process} on $M^S$ we mean an $M^S$-valued random variable 
(or a collection of $M$-valued random variables) 
whose (joint) distribution is invariant under the $\Ga$-action. 
An important class of invariant processes is \emph{factor of i.i.d.\ processes} defined as follows.  
Suppose that $Z_s$, $s\in S_1$, are independent and identically distributed 
$M_1$-valued random variables. We say that 
$Z = \left( Z_s \right)_{s \in S_1}$ is an i.i.d.\ process on $M_1^{S_1}$. 
Given a $\Ga$-factor $F \colon M_1^{S_1} \to M_2^{S_2}$,
$X \defeq F(Z)$ is a factor of the i.i.d.\ process $Z$.
It can be regarded as a collection of $M_2$-valued random variables: 
$X = \left( X_s \right)_{s \in S_2}$. 

In fact, all this can be viewed in the context of ergodic theory. 
An invariant process in the above sense gives rise to a dynamical system over $\Ga$: 
the group $\Ga$ acts by measure-preserving transformations 
on the measurable space $M^S$ equipped with a probability measure 
(the distribution of the invariant process). 
An i.i.d.\ process simply corresponds to a (generalized) Bernoulli shift. 
Therefore factor of i.i.d.\ processes are essentially factors of Bernoulli shifts. 
Classical ergodic theory ($\IZ$-factors) 
have the largest literature and the most complete theory 
but $\Ga$-factors have also been thoroughly investigated for general $\Ga$. 

For amenable group actions (the Kolmogorov-Sinai) entropy 
serves as a complete invariant (for isomorphism of Bernoulli shifts). 
As for the nonamenable case, Ornstein and Weiss asked whether 
all Bernoulli shifts are isomorphic over a nonamenable group \cite{ornstein_weiss_1987}. 
This remained open until the breakthrough results of Lewis Bowen: 
he answered the question negatively by introducing 
the \emph{$f$-invariant} for free group actions \cite{f-invariant} 
and the \emph{$\Sigma$-entropy} for actions of sofic groups \cite{sofic_entropy}. 
In another paper he showed that the $f$-invariant is 
essentially a special case of the $\Sigma$-entropy 
which has the consequence that the $f$-invariant is non-negative 
for factors of the Bernoulli shift \cite[Corollary 1.8]{bowen}. 
We will need this fact in the form of an entropy inequality, 
see \eqref{eq:free_group_entropy_inequality} below. 

\subsection{Factors on $T_d$}
\label{sec:factors_Td}
The main results of this paper (Theorem \ref{thm:universal_bound} and \ref{thm:fixed_process}) 
are concerned with factor of i.i.d.\ processes on $T_d$. 
This corresponds to the case when $\Ga$ is the automorphism group $\Aut(T_d)$ of
the $d$-regular infinite tree $T_d$ and $S$ is the vertex set $V(T_d)$. 

When we say \emph{factor of i.i.d.\ process}, 
we should also specify which i.i.d.\ process we have in mind
(that is, specify $M_1$ and a probability distribution on it). 
By default we will work with the uniform $[0,1]$ measure 
(i.e., the Lebesgue measure on $[0,1]$). 
In fact, as far as the class of factor processes is concerned, 
it does not really matter which i.i.d.\ process we consider. 
For example, for $\{0,1\}$ with the uniform distribution 
we get the same class of factors as for the uniform $[0,1]$ measure. 
This follows from the fact that these two i.i.d.\ processes 
are $\Aut(T_d)$-factors of each other \cite{karen_ball}. 

Note that a factor of i.i.d.\ process $X$ on $T_d$ is $\Aut(T_d)$-invariant. 
Therefore each $X_v$ has the same distribution. 
Moreover, the joint distribution of $X_u$ and $X_v$ 
(and hence their correlation or mutual information) 
depends only on the distance between $u$ and $v$. 

One of our goals in this paper is to find a universal upper bound 
for the mutual information $I(X_u; X_v)$ % = H(X_u)+H(X_v)-H(X_u,X_v)$ 
that holds for any factor of i.i.d.\ process $X$. 
The next example, where a tuple of independent copies of 
the same factor of i.i.d.\ process is considered, 
shows that this goal is plausible only if we normalize $I(X_u; X_v)$ in some way. 
That is why we introduced the \emph{normalized mutual information} 
$I(X_u; X_v) / H(X_v)$. 
\begin{example} \label{ex:tuple}
Given a factor of i.i.d.\ process $X = F(Z)$ with a finite state space $M$ 
there exists a factor of i.i.d.\ process $Y= (Y^1, \ldots, Y^n)$ 
with state space $M^n = M \times \cdots \times M$ such that each 
$Y^i = \left( Y^i_v \right)_{v \in V(T_d)}$ is an independent copy of $X$. 
(The point is that one can take $n$ independent copies $Z^1, \ldots, Z^n$ of the i.i.d.\ process $Z$ 
and apply $F$ to each $Z^i$ to get $Y^i$. It is easy to see that 
$(Z^1, \ldots, Z^n)$ can be obtained as a factor of $Z$. 
Therefore the process $Y$ is also a factor of $Z$.) 
If we take $n$ copies of $X$ as described above, 
then each entropy and mutual information gets multiplied by $n$. 
On the other hand, the normalized mutual information 
(corresponding to two given vertices $u$ and $v$) is the same for $X$ and $Y$. 
\end{example}

\subsection{$F_r$-factors}
\label{sec:freegroup-factors}
The other case that will be of particular interest 
for us is when $\Ga$ is the free group $F_r$ of some rank $r$. 
We can set $S = \Ga = F_r$ and consider the natural action of $F_r$ on itself. 
Similarly as for $\Aut(T_d)$-factors, 
we use the uniform $[0,1]$ measure for the i.i.d.\ process. 
Using other measures would result in the same class of factor processes. 

This is actually a broader class than the class of $\Aut(T_d)$-factors (for $d=2r$). 
If $d=2r$, we can think of $T_d$ as the Cayley graph of $F_r$ 
with respect to a symmetric generating set $\{ a_1^{\pm 1}, \ldots, a_r^{\pm 1} \}$. 
That is, $V(T_d) = F_r$ and a vertex $g$ is incident to vertices of the form $g a_i^{\pm 1}$. 
Then $F_r$ acts on $V(T_d) = F_r$ (from the left) via automorphisms of this Cayley graph. 
So if we identify the elements of $F_r$ with these automorphisms, 
then $F_r$ becomes a subgroup of $\Aut(T_d)$, and consequently 
being $\Aut(T_d)$-equivariant is a stronger condition than being $F_r$-equivariant. 
In other words, every $\Aut(T_d)$-factor is an $F_r$-factor as well. 

For a general $F_r$-factor of i.i.d.\ we only have $F_r$-invariance 
(but not necessarily $\Aut(T_d)$-invariance). 
It is still true that each $X_g$ has the same distribution. 
As for the distribution of edges, however, 
$( X_g, X_{g a_i^{\pm 1}} )$ might have different 
distributions for different $a_i^{\pm 1}$. 

The following entropy inequality, which plays a central role 
in our proof of Theorem \ref{thm:universal_bound}, 
easily follows from the fact that the $f$-invariant 
of a factor of a Bernoulli shift is non-negative \cite{bowen}. 
\begin{theorem} \label{thm:free_group_entropy_inequality} 
Let $\Ga = \left\langle a_1, \ldots, a_r \right\rangle$ be a free group of rank $r \geq 2$.  
If $X = \left( X_g \right)_{g \in \Ga}$ is 
a $\Ga$-factor of the i.i.d.\ process on $[0,1]^{\Ga}$, 
then for a fixed $g \in \Ga$ we have 
\begin{equation}
\label{eq:free_group_entropy_inequality}
\frac{1}{r} \sum_{i=1}^r H(X_g, X_{g a_i}) \geq \frac{2r-1}{r} H(X_g) ,
\end{equation}
or equivalently:
\begin{equation}
\label{eq:free_group_mut_inf}
\frac{1}{r} \sum_{i=1}^r \frac{ I(X_g; X_{g a_i}) }{ H(X_g) } \leq \frac{1}{r}  .
\end{equation}
\end{theorem}
\begin{remark} 
%Note that the inequality holds for any choice of the free generating set. 
This is more general than the edge-vertex entropy inequality 
\eqref{eq:edge-vertex} for $\Aut(T_d)$-factors for $d=2r$. 
Indeed, given an $\Aut(T_d)$-factor, it is also an $F_r$-factor, 
but with the extra property that the distributions of edges are the same. 
\end{remark}
%

%%%%%%%%%%%%%%%%%%%%%%%%%%%%%%%%%%%%%%%%%%%%%%%%%%%%%%%%%%%%%%%%%%%%%%%%%
%%%%%%%%%%%%%%%%%%%%%%%%%%%%%%%%%%%%%%%%%%%%%%%%%%%%%%%%%%%%%%%%%%%%%%%%%
%%%%%%%%%%%%%%%%%%%%%%%%%%%%%%%%%%%%%%%%%%%%%%%%%%%%%%%%%%%%%%%%%%%%%%%%%

\section{The universal bound} \label{sec:3}

In this section $G$ will denote the free product of 
$r$ copies of $\IZ$ and $t$ copies of $\IZ_2$ 
for different values of $r$ and $t$: 
$$ G = \underbrace{\IZ \ast \cdots \ast \IZ}_{r} \ast 
\underbrace{\IZ_2 \ast \cdots \ast \IZ_2}_{t} = 
\left\langle a_1, \ldots, a_r, a_{r+1}, \ldots, a_{r+t} \, | \, 
a_{r+1}^2 = \cdots = a_{r+t}^2 = e \right\rangle .$$
Let $A$ denote the set $\{ a_1^{\pm 1}, \ldots, a_r^{\pm 1}, a_{r+1}, \ldots, a_{r+t} \}$. 
First we define the \emph{word metric} on $G$ with respect to $A$. 
We will refer to the elements of $A$ as \emph{letters} 
and to products of these elements as \emph{words}. 
An element $g \in G$ can be represented by many words 
but for each $g$ there exists a unique shortest representing word. 
Actually, starting with any word representing $g$, 
by performing all possible cancellations in that product 
one always gets the shortest representing word 
that we will call the \emph{reduced form}. 
We define the \emph{length of $g$} as the length of this reduced form. 
(As for the \emph{unit element} $e$ of $G$, 
it is represented by the empty product, 
and hence the length of $e$ is $0$.) 

Note that the Cayley graph of $G$ with respect to $A$ is $T_d$ for $d=2r+t$. 
(That is, $V(T_d) = G$ and a vertex $g$ is incident to vertices of the form $g h$, $h \in A$.) 
The \emph{word metric} on $G$ (w.r.t.\ $A$) 
actually coincides with the graph distance on this Cayley graph. 

Our goal is to apply the inequality 
(\ref{eq:free_group_entropy_inequality}--\ref{eq:free_group_mut_inf}) 
for free subgroups of $G$. 
To obtain a result about vertices at distance $k$ in $T_d$ 
we will need a free subgroup $H$ that is generated by elements of length $k$. 
The higher the rank of our subgroup, the better inequality we get. 
Therefore we need to find as many elements of length $k$ as possible 
such that they freely generate a subgroup. (Although we will not need this fact, 
we mention that when we have the maximal possible number of elements, 
the generated subgroup has finite index.) 
\begin{lemma} \label{lem:subgroup1}
Let $d=2r$ and let 
$$ G = F_r = \underbrace{\IZ \ast \cdots \ast \IZ}_{r} = 
\left\langle a_1, \ldots, a_r \right\rangle .$$ 
Then for any odd integer $k = 2l+1$ there exists a free subgroup $H \leq G$ 
of rank $d(d-1)^l/2$ that is generated freely by elements of length $k$ 
(in the corresponding word metric). 
\end{lemma}
\begin{lemma} \label{lem:subgroup2}
Let $d=2r+1$ and let 
$$ G = F_r \ast \IZ_2 = \underbrace{\IZ \ast \cdots \ast \IZ}_{r} \ast \IZ_2 = 
\left\langle a_1, \ldots, a_r, a_{r+1} \, | \, a_{r+1}^2 = e \right\rangle .$$ 
Then for any odd integer $k = 2l+1$ with $l \geq 1$ there exists a free subgroup $H \leq G$ 
of rank $d(d-1)^l/2$ that is generated freely by elements of length $k$ 
(in the corresponding word metric). 
\end{lemma}
\begin{lemma} \label{lem:subgroup3}
Let $d \geq 3$ be arbitrary and let 
$$ G = \underbrace{\IZ_2 \ast \cdots \ast \IZ_2}_{d} = 
\left\langle a_1, \ldots, a_d \, | \, a_1^2 = \cdots = a_d^2 = e \right\rangle .$$ 
Then for any even integer $k = 2l$ there exists a free subgroup $H \leq G$ 
of rank $(d-1)^l$ that is generated freely by elements of length $k$ 
(in the corresponding word metric). 
\end{lemma}
Before we prove the above lemmas, let us show how 
Theorem \ref{thm:universal_bound} follows. 
We start with a technical lemma. 
\begin{lemma} \label{lem:technical}
Suppose that $H$ is a subgroup of a countable group $G$. 
Let us equip the spaces $[0,1]^H$ and $[0,1]^G$ 
with the product of uniform $[0,1]$ measures. 
Then there exists a $[0,1]^H \to [0,1]^G$ mapping that is 
measure-preserving and $H$-equivariant. 
\end{lemma}
\begin{proof}
Let us fix measure-preserving mappings $\varphi \colon [0,1] \to \{0,1\}^{\IN}$ 
and $\psi \colon \{0,1\}^{\IN} \to [0,1]$, where $\{0,1\}$ is equipped with 
the (discrete) uniform distribution. 

Let us also fix a set $T$ that contains exactly one element of each right $H$-coset, 
meaning that $(h,t) \mapsto ht$ defines a bijection $H \times T \to G$. 
Using the trivial $H$-action on $T$ and the natural (left) $H$-actions on $H$ and $G$ 
the above bijection will clearly be $H$-equivariant. 
This induces an $H$-equivariant mapping 
$\alpha \colon \{0,1\}^{H \times T} \to \{0,1\}^G$. 

Since $T$ is either finite or countably infinite, 
$T \times \IN$ has the same cardinality as $\IN$, 
so we can fix a bijection between these sets as well. 
This bijection yields a measure-preserving mapping 
$\beta \colon \{0,1\}^{\IN} \to \{0,1\}^{T \times \IN}$. 

Combining the above mappings we get the following: 
$$ [0,1]^H \xrightarrow{\varphi \times \varphi \times \cdots} 
\{0,1\}^{H \times \IN} \xrightarrow{\beta \times \beta \times \cdots} 
\{0,1\}^{H \times T \times \IN} \xrightarrow{\alpha \times \alpha \times \cdots}  
\{0,1\}^{G \times \IN} \xrightarrow{\psi \times \psi \times \cdots} [0,1]^{G} .$$
Each of the above mappings clearly preserves measure and commutes with the $H$-actions.
\end{proof} 
\begin{proof}[Proof of Theorem \ref{thm:universal_bound}]
For $k=1$ the statement of the theorem is equivalent to \eqref{eq:mut_info} 
so we may assume that $k \geq 2$. Depending on $d$ and $k$ we choose 
the group $G$ and the positive integer $r'$ as follows:
\begin{align*} 
&\mbox{if $k=2l+1 \geq 3$ is odd and $d=2r$ is even: } & 
&G=\underbrace{\IZ \ast \cdots \ast \IZ}_{r}, & r'&=d(d-1)^l/2; \\
&\mbox{if $k=2l+1 \geq 3$ is odd and $d=2r+1$ is odd: } & 
&G=\underbrace{\IZ \ast \cdots \ast \IZ}_{r} \ast \IZ_2, & r'&=d(d-1)^l/2; \\
&\mbox{if $k=2l$ is even and $d$ is arbitrary: } & 
&G=\underbrace{\IZ_2 \ast \cdots \ast \IZ_2}_{d}, & r'&=(d-1)^l. 
\end{align*}
Let $A \subset G$ still denote the generating set described at the beginning of this section. 
Recall that the Cayley graph of $G$ with respect to $A$ is $T_d$ 
so from this point on $V(T_d)$ is identified with $G$. 
According to Lemma \ref{lem:subgroup1}--\ref{lem:subgroup3} 
in each of the above cases $G$ has a free subgroup $H$ of rank $r'$ 
such that $H$ has a free generating set $S_0$ consisting of elements of length $k$ 
(in the word metric of $G$ with respect to $A$). 

Now let $X = \left( X_v \right)_{v \in G}$ be a factor of i.i.d.\ process 
over $V(T_d)=G$ with a finite state space $M$. 
This means that there exists an $\Aut(T_d)$-factor mapping $F \colon [0,1]^{G} \to M^{G}$ 
such that $X=F(Z)$ where $Z$ is the i.i.d.\ process on $[0,1]^{G}$. 
According to Lemma \ref{lem:technical} there exist an $H$-equivariant 
mapping $\varrho \colon [0,1]^H \to [0,1]^G$ such that $Z = \varrho(\tilde{Z})$ 
where $\tilde{Z}$ is an i.i.d.\ process on $[0,1]^H$. 

By $\pi_H$ we denote the projection $M^{G} \to M^{H}$. 
We have the following situation:
$$ [0,1]^H  \xrightarrow{\varrho}  [0,1]^{G} \xrightarrow{F} M^{G} \xrightarrow{\pi_H} M^H, $$
where all three mappings are $H$-equivariant, 
and hence their composition is an $H$-factor mapping. 
This means that if we consider $X$ over $H$, 
then we get an $H$-factor of i.i.d.\ process: 
$\left( X_h \right)_{h \in H} = \pi_H \circ F \circ \varrho ( \tilde{Z} )$. 
Therefore we can apply \eqref{eq:free_group_mut_inf} to $H$ 
and its free generating set $S_0$ of size $r'$. 
For any $h\in H$ and any $s \in S_0$, 
the vertices $h$ and $hs$ have distance $k$ (in the graph metric of $T_d$). 
Then, because of the $\Aut(T_d)$-invariance of $X$, 
the normalized mutual information $I(X_h; X_{hs}) / H(X_h)$ 
is the same for all $h$ and $s$. 
Therefore in our case the average on the left-hand side of \eqref{eq:free_group_mut_inf} 
is simply equal to $I(X_u; X_v) / H(X_v)$ for any $u,v \in V(T_d)$ with $\dist(u,v)=k$, 
while the right hand side is $1/r'$, and hence Theorem \ref{thm:universal_bound} follows. 
(The sharpness will be shown in Section \ref{sec:5}.) 
\end{proof}
It remains to prove Lemma \ref{lem:subgroup1}--\ref{lem:subgroup3}.
\begin{proof}[Proof of Lemma \ref{lem:subgroup1}]
The set of letters in this case is $A = \{ a_1^{\pm 1}, \ldots, a_r^{\pm 1} \}$. 
A word is called a \emph{palindrome} if it reads the same backward as forward. 
Let us consider the following set of words: 
$$ S \defeq \left\{ s \in G : 
\mbox{the reduced form of $s$ is a palindrome and has length $2l+1$} \right\} .$$
That is, elements of $S$ are in the form $b_1 \cdots b_l b_{l+1} b_l \cdots b_1$, 
where $b_i \in A$ %for each $i=1,\ldots,l+1$ 
and $b_{i+1} \neq b_i^{-1}$. %for each $i=1, \ldots, l$. 
The number of such elements is clearly $2r(2r-1)^l = d(d-1)^l$. 

The inverse of a palindrome is also a palindrome 
(and not the same palindrome because $G$ has no elements of order $2$). 
Therefore there exists $S_0 \subset S$ with $|S_0| = |S|/2= d(d-1)^l/2$ such that 
$S = S_0 \cup S_0^{-1}$ where $S_0^{-1} = \{ s^{-1} \, : \, s \in S_0 \}$.  
We will see that $S_0$ is a free generating set of a subgroup 
$H \leq G$ that has all the required properties. 

The key observation is the following. 
\begin{claim*}
Let $s_1, \ldots, s_n$ be palindromes in $S$ such that $s_{i+1} \neq s_i^{-1}$ for each $i$. 
Then the reduced form of the product $s_1 \cdots s_n$ has length at least $2l+n$ 
and its last $l+1$ letters are the same as those of $s_n$.
\end{claim*}
We prove the claim by induction. It is obvious for $n=1$. For $n \geq 1$ let us assume that 
the reduced form of the product $s_1 \cdots s_n$ ends with the same $l+1$ letters as $s_n$ 
and let $s_{n+1} \neq s_n^{-1}$. This means that when we multiply 
the reduced form of $s_1 \cdots s_n$ by $s_{n+1}$ at most $l$ letters will be cancelled out 
and the remaining (at least $l+1$) letters of $s_{n+1}$ will appear unchanged at the end of the product. 
It follows that the last $l+1$ letters of the reduced form of $s_1 \cdots s_n s_{n+1}$ 
will be the same as those of $s_{n+1}$. We also get that $s_1 \cdots s_n s_{n+1}$ 
is at least $(l+1)-l=1$ longer than $s_1 \cdots s_n$, which completes the induction. 

In particular, the product $s_1 \cdots s_n$ cannot be the unit element of $G$. 
Therefore $S_0$ freely generates some subgroup $H \leq G$, 
the rank of which is, obviously, $|S_0| = d(d-1)^l/2$, 
and this is what we wanted to prove. 

In fact, $H$ has finite index. (We do not need this property in this paper.) 
This follows from the following observation. 
Let $T \subset G$ denote the set of elements of length at most $l$. 
Then it is easy to see that every element of $G$ can be (uniquely) written 
in the form $s_1 \cdots s_n t$, where $t \in T$, $s_i \in S$ and $s_{i+1} \neq s_i^{-1}$.
\end{proof}
\begin{proof}[Proof of Lemma \ref{lem:subgroup2}]
Essentially the same proof works. 
Here the set of letters is $A = \{ a_1^{\pm 1}, \ldots, a_r^{\pm 1}, a_{r+1} \}$, 
and one of the letters ($a_{r+1}$) has order $2$ meaning $a^{-1}_{r+1} = a_{r+1}$. 
However, we can still define the set $S$ of palindromes of length $k=2l+1$ 
for which we have $|S| = (2r+1)(2r)^l = d(d-1)^l$. 
The same claim as in the previous proof remains true. 
The only difference is that in this case $G$ has elements of order $2$. 
So we need to check that $S$ contains no element of order $2$, 
which is clearly true unless $l = 0$. 
The rest of the proof is the same. 
\end{proof}
\begin{proof}[Proof of Lemma \ref{lem:subgroup3}]
In this lemma the set of letters $A = \{a_1, \ldots, a_d \}$ 
consists of elements of order $2$. 
It is an easy exercise that for $l=1$ the set 
$$ B_0 \defeq \left\{ a_i a_1 : 2 \leq i \leq d \right\} $$ 
is a free generating set (of size $d-1$) of the subgroup of $G$ 
consisting of all elements of even length. Note that 
$$ B_0^{-1} = \left\{ a_1 a_i : 2 \leq i \leq d \right\} .$$

For $l \geq 2$ we will need to nest the $d-1$ elements of $B_0$ 
in palindrome-like words of length $2l$. 
First we define the mappings $\varphi_j \colon A \to A$: 
for $j \in \{1,\ldots, d-1\}$ let $\varphi_j$ shift the indices by $j$, 
that is, $\varphi_j( a_i ) \defeq a_{i+j}$. 
(The addition in the index is meant modulo $d$.) 
We will consider words of the following form: 
for any given $j \in \{1,\ldots, d-1\}$ and any given sequence of 
letters $b_1, \ldots, b_l$ from $A$ such that $b_1 = a_1$ and $b_{i+1} \neq b_i$ 
take the word 
$$ \varphi_j(b_l) \cdots \varphi_j(b_2) 
\underbrace{\varphi_j(b_1)}_{a_{j+1}} 
\underbrace{b_1}_{a_1} b_2 \cdots b_l .$$
Note that these words have length $2l$ and 
for the two letters in the middle we have 
$\varphi_j(b_1) b_1 = a_{j+1} a_1 \in B_0$. 
We claim that the set $S_0$ of these $(d-1)^l$ words freely generates a subgroup. 

The following is straightforward by induction. 
\begin{claim*}
Let $s_1, \ldots, s_n$ be words in $S \defeq S_0 \cup S_0^{-1}$ 
such that $s_{i+1} \neq s_i^{-1}$ for each $i$. 
Then the product $s_1 \cdots s_n$ has the following property:
\begin{itemize}
\item if $s_n \in S_0^{-1}$, then the last $l$ letters 
in the reduced form of $s_1 \cdots s_n$ are the same as in $s_n$;
\item if $s_n \in S_0$, then the last $l+1$ letters 
in the reduced form of $s_1 \cdots s_n$ are the same as in $s_n$.
\end{itemize}
\end{claim*}
It immediately follows that the length of the reduced form of 
the product $s_1 \cdots s_n$ cannot decrease 
when multiplied by a new element $s_{n+1} \neq s_n^{-1}$. 
In particular, for $n \geq 1$ the product $s_1 \cdots s_n$ 
cannot be equal to the unit element $e$. 
Therefore $S_0$ freely generates a subgroup of rank $|S_0| = (d-1)^l$. 
\end{proof}

%%%%%%%%%%%%%%%%%%%%%%%%%%%%%%%%%%%%%%%%%%%%%%%%%%%%%%%%%%%%%%%%%%%%%%%%%
%%%%%%%%%%%%%%%%%%%%%%%%%%%%%%%%%%%%%%%%%%%%%%%%%%%%%%%%%%%%%%%%%%%%%%%%%
%%%%%%%%%%%%%%%%%%%%%%%%%%%%%%%%%%%%%%%%%%%%%%%%%%%%%%%%%%%%%%%%%%%%%%%%%

\section{The rate of decay for a fixed process} \label{sec:4}

We will need three ingredients to prove Theorem \ref{thm:fixed_process}. 
The first one is a bound for the correlation 
of a pair of vertices for factor of i.i.d.\ processes on $\IR^{V(T_d)}$,  
which was proved by Backhausz, Szegedy, and Vir\'ag in \cite{cordec}: 
\begin{equation} \label{eq:corr_decay_for_vertices}
\left| \corr( X_u, X_v ) \right| \leq
\left( k+1 - \frac{2k}{d} \right) \left( \frac{1}{ \sqrt{d-1} } \right)^k
\mbox{, where } k = \dist(u,v) ,
\end{equation}
that is, the rate of the correlation decay is essentially $1 / (\sqrt{d-1})^k$.
(Here it is assumed that $\var X_v < \infty$.)

Now suppose we have a finite state space $M$ and a factor of i.i.d.\ process on $M^{V(T_d)}$. 
How can we make use of the above result in this case? 
Taking any function $f \colon M \to \IR$ we can replace each $X_v$ with $f(X_v)$ 
to get a factor of i.i.d.\ on $\IR^{V(T_d)}$ so that 
\eqref{eq:corr_decay_for_vertices} can be applied. 
The second ingredient is the next lemma from \cite{one_ended_tail} 
which tells us that the same bound holds 
if we take different real-valued functions of $X_u$ and $X_v$. 
\begin{lemma} \label{lem:2functions}
Let $(A, \mathcal{F} )$ be an arbitrary measurable space.
Suppose that the $(A, \mathcal{F} )$-valued random variables $X_1,X_2$ are exchangeable
(that is, $(X_1,X_2)$ and $(X_2,X_1)$ have the same joint distribution),
and that there exists a constant $\al \geq 0$ with the property
that for any measurable $f \colon A \to \IR$ we have
\begin{equation} \label{eq:1function}
\left| \corr\big( f(X_1), f(X_2) \big) \right| \leq \al
\mbox{ provided that $f(X_1)$ has finite variance.}
\end{equation}
Then for any measurable functions $f_1,f_2 \colon A \to \IR$
\begin{equation} \label{eq:2functions}
\left| \corr\big( f_1(X_1), f_2(X_2) \big) \right| \leq \al
\mbox{ provided that $f_1(X_1)$ and $f_2(X_2)$ have finite variances.}
\end{equation}
\end{lemma}
\begin{proof}
The detailed proof can be found in \cite[Lemma 3.2]{one_ended_tail}. 
We include a sketch here for the sake of completeness. 
After rescaling we might assume that $\var( f_1(X_1) ) = \var( f_2(X_2) ) = 1$. 
If we apply \eqref{eq:1function}
to the function $f= f_1+f_2$ and also to $f=f_1-f_2$, 
we reach \eqref{eq:2functions} after a short and simple calculation. 
Note that the exchangeability of $X_1$ and $X_2$ implies 
$ \cov( f_1(X_1), f_2(X_2) ) = \cov( f_1(X_2), f_2(X_1) )$.  
\end{proof}
The final ingredient is the following lemma 
linking correlation to mutual information. 
\begin{lemma} \label{lem:corr_and_mut_info}
Let $X,Y$ be discrete random variables. 
Suppose that there exists a real number $\al \geq 0$ such that 
for any (real-valued) functions $f(X)$ and $g(Y)$ of $X$ and $Y$ it holds that 
$ \left| \corr\big( f(X), g(Y) \big) \right| \leq \al$. 
Then we have 
$$ I(X;Y) = H(X) - H(X | Y) \leq (m-1) \al^2 ,$$
where $m$ denotes the number of values $X$ can take. 
\end{lemma}
\begin{proof}
Let $A$ be an event that depends on $X$, that is, $\ind_A = f(X)$ for some function $f$. 
We denote the probability $\P(A)$ by $p$ and we set  
$$ g_A(y) \defeq \P( A | Y=y ) - \P(A) = \P( A | Y=y ) - p .$$
Clearly, $\E g_A(Y)=0$, and it is also easy to see that 
$$ \corr\big( f(X), g_A(Y) \big) = \frac{ \sqrt{\E g_A(Y)^2} }{\sqrt{p(1-p)}} .$$ 
It follows that 
\begin{equation} \label{eq:var}
\E g_A(Y)^2 \leq \al^2 p(1-p) .
\end{equation}
Now let us assume that $X$ takes the value $x_i$ with probability $p_i$ for $1 \leq i \leq m$. 
We will need to use the above inequality for each event $A_i = \ind_{\{X=x_i\}}$, $1\leq i \leq m$. 
We write $g_i$ for the corresponding function $g_{A_i}$. 

According to \eqref{eq:conditional_entropy} 
the conditional entropy $H(X | Y)$ can be expressed as 
$$ - H(X | Y) = \E \sum_{i=1}^m (p_i+g_i(Y)) 
\underbrace{\log(p_i+g_i(Y))}_{\log(p_i)+ \log \left( 1 + \frac{g_i(Y)}{p_i} \right) } .$$ 
Now by using the inequality $\log(1+x) \leq x$ we get that 
$$ -H(X | Y) \leq \underbrace{\sum_{i=1}^m p_i \log(p_i)}_{-H(X)} + 
\sum_{i=1}^m \E g_i(Y) \log(p_i) + 
\sum_{i=1}^m \E \left( \Big( p_i+g_i(Y) \Big) \frac{g_i(Y)}{p_i} \right).$$
Using that $\E g_i(Y)=0$ we conclude that 
$$ I(X;Y) = H(X) - H(X | Y) \leq \sum_{i=1}^m \E \frac{g_i(Y)^2}{p_i} \leq 
\al^2 \sum_{i=1}^m (1-p_i) = (m-1) \al^2 ,$$
where the last inequality follows from \eqref{eq:var}. 
\end{proof}
\begin{remark}
Although we will not need it in this generality, 
we mention that the lemma is true even when 
only one of the two random variables is assumed to be discrete. 
Let $X$ be discrete and $Y$ arbitrary, 
and suppose that $ \left| \corr\big( f(X), g(Y) \big) \right| \leq \al$ 
for any $f$ and any measurable $g$. Then it still follows that 
$H(X) - H(X | Y) \leq (m-1) \al^2$. 

The point is that one can use \eqref{eq:conditional_entropy} 
to define the conditional entropy $H(X|Y)$ even when $Y$ is not discrete: 
for an event $A$ the mapping $y \mapsto \P( A | Y=y )$ 
needs to be replaced by the conditional expectation $\E( \ind_A | Y )$, 
which is a measurable function of $Y$. 
The same modification needs to be made in the above proof. 
\end{remark}

Now we have all the ingredients to prove Theorem \ref{thm:fixed_process}.
\begin{proof}[Proof of Theorem \ref{thm:fixed_process}]
Let $M$ be finite and let $X_v$, $v \in V(T_d)$, be a factor of i.i.d.\ process on $M^{V(T_d)}$. 
Suppose that the distance of the vertices $u$ and $v$ is $k$ and set 
$$ \alpha = \frac{k+1}{\left( \sqrt{d-1} \right)^k } .$$ 
Then by \eqref{eq:corr_decay_for_vertices} we know that 
$\left| \corr( f(X_u) , f(X_v) ) \right| \leq \alpha$ for any function $f \colon M \to \IR$. 
There is an automorphism of $T_d$ taking $u$ to $v$ and $v$ to $u$, 
which means that the random variables $X_u$, $X_v$ are exchangeable. 
Therefore we can apply Lemma \ref{lem:2functions} to $X_u$ and $X_v$ 
and we obtain that $\left| \corr( f(X_u) , g(X_v) ) \right| \leq \alpha$ 
for any functions $f,g \colon M \to \IR$. By Lemma \ref{lem:corr_and_mut_info} 
it follows that 
$$ I( X_u; X_v ) < |M| \alpha^2 = \frac{|M| (k+1)^2}{(d-1)^k} ,$$
and this is exactly what we wanted to prove.
\end{proof}
%

%%%%%%%%%%%%%%%%%%%%%%%%%%%%%%%%%%%%%%%%%%%%%%%%%%%%%%%%%%%%%%%%%%%%%%%%%
%%%%%%%%%%%%%%%%%%%%%%%%%%%%%%%%%%%%%%%%%%%%%%%%%%%%%%%%%%%%%%%%%%%%%%%%%
%%%%%%%%%%%%%%%%%%%%%%%%%%%%%%%%%%%%%%%%%%%%%%%%%%%%%%%%%%%%%%%%%%%%%%%%%

\section{Examples} \label{sec:5}

In this section we construct factor of i.i.d.\ processes 
showing that our bounds are (essentially) sharp.

\subsection{Sharpness of Theorem \ref{thm:universal_bound}}

Let $k$ be a fixed positive integer and $u,v \in V(T_d)$ vertices at distance $k$. 
We claim that there exist factor of i.i.d.\ processes $X$ on $T_d$ such that 
the normalized mutual information $I(X_u;X_v)/H(X_v)$ 
can be arbitrarily close to the upper bound 
\begin{equation} \label{eq:bound}
\beta_k \defeq 
\begin{cases}
\frac{2}{d(d-1)^l} & \mbox{ if $k=2l+1$ is odd,} \\
\frac{1}{(d-1)^l} & \mbox{ if $k=2l$ is even.}
\end{cases}
\end{equation}
The idea is the following: given i.i.d.\ labels at each vertex, 
let the factor process list all the labels 
within some large distance $R$ at any given vertex. 
When we look at the joint distribution of $X_u$ and $X_v$ 
we get a collection of i.i.d.\ labels with some labels listed twice. 
Hence the normalized mutual information is $| B_R(u) \cap B_R(v) | / | B_R(v) |$, 
where $B_R(v)$ denotes the ball of radius $R$ around $v$. 
It is easy to see that this converges to $\beta_k$ as $R \to \infty$. 

For a rigorous argument we need to be more careful 
since listing the labels should be done in an $\Aut(T_d)$-invariant way. 
We first introduce two auxiliary lemmas and then precisely define our example.
\begin{lemma} \label{lm:max_sparse_set}
For any positive integer $L$ there exists a factor of i.i.d.\ $0$-$1$ labeling 
of the vertices of $T_d$ such that any ball of radius $L$ contains a vertex with label $1$ 
but any two vertices of label $1$ have distance greater than $L$. 
\end{lemma}
\begin{lemma} \label{lm:max_sparse_coloring}
For any positive integer $L$ there exists a factor of i.i.d.\ coloring  
of the vertices of $T_d$ such that finitely many colors are used 
and vertices of the same color have distance greater than $L$. 
\end{lemma}
\begin{example} \label{ex:universal_bound}
Given $k$ and $R$, let $C= \left( C_w \right)_{w \in V(T_d)}$ 
be a factor of i.i.d.\ coloring provided by Lemma \ref{lm:max_sparse_coloring} 
for $L = 2R+k$. For a positive integer $N$ 
let $Z_w$, $w \in V(T_d)$, be i.i.d.\ uniform labels on $\{1,2,\ldots,N\}$. 
We set 
$$X_v = \{(C_w,Z_w)~|~w \in B_R(v) \}.$$
Then for vertices $u,v$ at distance $k$ we have 
\begin{equation} \label{eq:mut_inf_appr}
\frac{I(X_u;X_v)}{H(X_v)} = 
\frac{| B_R(u) \cap B_R(v) | }{ | B_R(v) | } + o_N(1) .
\end{equation}
\end{example}
Indeed, $X_v$ can be viewed as the list of
variables $(C_w, Z_w)$, $w\in B_R(v)$, ordered by $C_w$ (which
are all different). This clearly defines an 
$\Aut(T_d)$-factor of i.i.d.\ process $X_v$, $v \in V(T_d)$. 
Conditioned on the coloring process $C$, 
the entropies are easy to compute: 
$$ H( X_v | C ) = | B_R(v) | \log N \quad \mbox{ and } \quad 
H( X_u, X_v | C ) = | B_R(u) \cup B_R(v) | \log N .$$
Since the contribution of the coloring to the entropies 
does not depend on $N$, it gets negligible when $N$ is large enough, 
and \eqref{eq:mut_inf_appr} follows. 

Finally, we prove the two lemmas.
\begin{proof}[Proof of Lemma \ref{lm:max_sparse_set}]
We describe the labeling as the output of a randomized local algorithm,
which is easy to interpret as a factor of i.i.d.\ process. 

In the beginning all labels are undefined. 
The algorithm consists of countably many steps. 
At every odd step every vertex with undefined label
proposes to get a label $1$ with probability $1/2$. 
Suppose that a vertex $v$ proposes to get label $1$. 
If there is no other vertex within distance $L$ of $v$ 
that also proposes to get label $1$, 
then the label of $v$ is fixed, otherwise the proposed label is withdrawn. 
At even steps, undefined vertices check if a label $1$ 
has appeared within distance $L$ and set their own label $0$ if this is the case. 

Note that at an odd step any undefined label gets fixed with probability 
greater than some positive constant $\eps$ depending on $L$.  
It follows that after countably many steps 
all labels will be defined with probability $1$. 
It is easy to verify that the obtained labeling has all the required properties. 
\end{proof}
\begin{proof}[Proof of Lemma \ref{lm:max_sparse_coloring}]
Lemma \ref{lm:max_sparse_set} is used to find vertices with color $1$. A
similar algorithm is applied for color $2$, but now
some vertices already have defined labels when launching the
algorithm. We continue by adding more colors the same way.

After having added $n$ colors this way, 
every ball of radius $L$ around an uncolored vertex 
must contain vertices of each color $1,2,\ldots,n$. 
When $n$ becomes equal to the number of vertices in a ball of radius $L$, 
this is not possible any longer, 
therefore we cannot have any more uncolored vertices at that point, 
meaning that we have colored all vertices 
in the required manner using at most $n$ colors. 
\end{proof}

\subsection{Sharpness of Theorem \ref{thm:fixed_process}}
The next example shows that the bound obtained 
in Theorem \ref{thm:fixed_process} is essentially sharp. 
First we briefly describe the construction. 
We start with an i.i.d.\ process where each label has standard normal distribution. 
Then we take a \emph{linear factor}: each new label is 
some linear combination of the i.i.d.\ labels. 
We will choose the coefficients in a way that the correlation decay 
for the obtained factor process is close to the bound \eqref{eq:corr_decay_for_vertices}. 
Then we take the sign of the label of this factor process at every vertex. 
We will see that for this $\{ \pm 1 \}$-valued process the correlation decays at roughly the same rate. 
However, for symmetric binary variables the mutual information 
is essentially the square of the correlation. 

More precisely, for any $\eps > 0$ we construct a factor of i.i.d.\ process (with two states) 
such that the mutual information for distance $k$ is $\Omega\left( k^{2-\eps} (d-1)^{-k} \right)$. 
\begin{example} \label{ex:fixed_process}
Fix a parameter $\eps>0$. 
Let $Z_w$, $w \in V(T_d)$, be i.i.d.\ standard normal random variables. 
We first define a factor $Y$ of the i.i.d.\ process $Z$ 
by taking linear combinations of $Z_w$ with the following coefficients: 
$$ Y_v \defeq \sum_{w \in V(T_d)} \alpha_{\dist(v,w)} Z_w \mbox{, where } 
\alpha_0=0 \mbox{ and } \alpha_k = \frac{ k^{-\frac{1}{2}-\eps} }{ \sqrt{d-1}^k } \mbox{ for } k \geq 1 .$$
Then apply the sign function at each vertex: 
$$ X_v \defeq \sign(Y_v) .$$ 
\end{example}

Note that $Y_v$ is well defined since the sum of the squares of the coefficients is finite. 
Therefore $Y_v$ is a normal random variable with mean $0$ 
and some positive and finite variance $\gamma = \gamma(\eps)$. 
From this point on $\gamma$ will denote a positive constant 
that depends only on $\eps$ (possibly a different constant at each occurrence). 

Suppose that $u$ and $v$ have distance $k$. 
We denote the unique path connecting them by $u_0=u, u_1, \ldots, u_{k-1}, u_k=v$. 
If we are at vertex $u_j$, $1\le j \le k-1$, and move distance $n$ away from the path, 
then we get to a vertex $w$ for which $\dist(u,w) = j+n$ and $\dist(v,w) = k-j+n$. 
The number of such vertices is clearly $(d-2)(d-1)^{n-1}$. Thus 
$$\cov(Y_u,Y_v) = \sum_{w \in V(T_d)} \al_{\dist(u,w)} \al_{\dist(v,w)}
\geq \frac{\gamma}{\sqrt{d-1}^{k}} \sum_{j=1}^{k-1} \sum_{n=1}^\infty 
(j+n)^{-\frac{1}{2} -\eps} (k-j+n)^{-\frac{1}{2} -\eps} .$$
We ignore the terms for which $j+n < k$ and 
rearrange the rest of the sum grouping the terms based on the value $m \defeq j+n$. 
For a given $m \geq k$ and $j \in \{1, \ldots, k-1\}$ 
we have $n=m-j$ and hence $k-j+n = k+m-2j$. 
Therefore the average of $k-j+n$ for a given $m$ as $j$ runs through $1, \ldots, k-1$ 
is exactly $m$, and consequently the convexity of $x^{-\frac{1}{2}-\eps}$ implies that  
$$ \sum_{j=1}^{k-1} (k+m-2j)^{-\frac{1}{2}-\eps} \ge (k-1) m^{-\frac{1}{2}-\eps} .$$
It follows that 
$$ \cov(Y_u,Y_v) \ge \frac{\gamma (k-1)}{\sqrt{d-1}^{k}} 
\sum_{m=k}^{\infty} m^{-1-2\eps} \ge \frac{\gamma (k-1)}{\sqrt{d-1}^{k}} 
\underbrace{ \int_{k}^\infty x^{-1-2\eps} \, \mathrm{d}x }_{ k^{-2\eps} /(2\eps) }
\ge \frac{\gamma k^{1-2\eps}}{\sqrt{d-1}^{k}} ,$$
and the same is true for $\corr(Y_u,Y_v)$ (again with a different $\gamma$). 
Note that there exist constants $0<\gamma_1<\gamma_2$ such that 
for any $W,W'$ jointly normal random variables we have
$$ \gamma_1 \big| \corr(W, W') \big| \le 
\big| \corr(\sign(W), \sign(W')) \big| \le 
\gamma_2 \big| \corr(W, W') \big| .$$
This means that we get the same correlation 
(up to a constant factor) after taking the sign of $Y$:
$$\corr(X_u,X_v)\ge \frac{\gamma k^{1-2\eps}}{\sqrt{d-1}^{k}}.$$
Now working with symmetric binary variables, elementary computations
show that when $P(X_u=X_v)$ is close to $1/2$, we have
$$ \gamma_1 \left| P(X_u= X_v) - \frac{1}{2} \right| \leq 
\big| \corr(X_u, X_v) \big| \leq 
\gamma_2 \left| P(X_u= X_v) - \frac{1}{2} \right| ,$$
and
$$ \gamma_1 \left| P(X_u= X_v) - \frac{1}{2} \right|^2 \leq 
\big| I(X_u; X_v) \big| \leq 
\gamma_2 \left| P(X_u= X_v) - \frac{1}{2} \right|^2 $$
for some constants $0< \gamma_1 < \gamma_2$. It follows that 
$$I(X_u;X_v) \ge \frac{\gamma k^{2-4\eps}}{(d-1)^{k}} ,$$
which indeed confirms that the bound 
in Theorem \ref{thm:fixed_process} is essentially sharp.

%%%%%%%%%%%%%%%%%%%%%%%%%%%%%%%%%%%%%%%%%%%%%%%%%%%%%%%%%%%%%%%%%%%%%%%%%
%%%%%%%%%%%%%%%%%%%%%%%%%%%%%%%%%%%%%%%%%%%%%%%%%%%%%%%%%%%%%%%%%%%%%%%%%
%%%%%%%%%%%%%%%%%%%%%%%%%%%%%%%%%%%%%%%%%%%%%%%%%%%%%%%%%%%%%%%%%%%%%%%%%

\bibliographystyle{plain}
\bibliography{refs}

\end{document}